\newif\ifpreprint
\preprinttrue  

\ifpreprint

\documentclass[12pt]{article}
\usepackage{fullpage}

\title{End-to-End Learning to Warm-Start for \\Real-Time Quadratic Optimization}
\author{Rajiv Sambharya$^1$, Georgina Hall$^2$, Brandon Amos$^3$, and Bartolomeo Stellato$^1$}
\date{%
    $^1$Princeton University\\%
    $^2$INSEAD\\
    $^3$Meta AI\\[2ex]%
    \today
}

\usepackage[round,authoryear]{natbib}
\usepackage{amsmath,enumitem,mathtools,amsthm,amssymb}

\renewcommand\arraystretch{1.2}
\newtheorem{theorem}{Theorem}
\newtheorem{lemma}[theorem]{Lemma}
\newtheorem{corollary}[theorem]{Corollary}
\usepackage{hyperref}
\newcommand{\acks}[1]{\section*{Acknowledgments}#1}
\usepackage[algo2e,ruled]{algorithm2e}

\else

\documentclass{l4dc2023}

\renewcommand\arraystretch{1.2}

\title[End-to-End Learning to Warm-Start QPs]{End-to-End Learning to Warm-Start \\for Real-Time Quadratic Optimization}

\author{%
 \Name{Rajiv Sambharya} \Email{rajivs@princeton.edu}\\
 \addr Operations Research and Financial Engineering, Princeton University, Princeton, NJ, USA
 \AND
 \Name{Georgina Hall} \Email{georgina.hall@insead.edu}\\
 \addr Decision Sciences, INSEAD, Fontainebleau, France
 \AND
 \Name{Brandon Amos} \Email{bda@fb.com}\\
 \addr Meta AI, New York City, NY, USA
 \AND
 \Name{Bartolomeo Stellato} \Email{bstellato@princeton.edu}\\
 \addr Operations Research and Financial Engineering, Princeton University, Princeton, NJ, USA
}

\usepackage{caption}
\fi

\usepackage{todonotes}

\definecolor{bdacolor}{RGB}{168, 141, 201}

\definecolor{ghcolor}{RGB}{48,213,200}

\usepackage{verbatim}
\usepackage{multicol}
\usepackage{algorithm}
\usepackage[noend]{algpseudocode} 
\usepackage{csvsimple}	
\usepackage{booktabs}   
\usepackage[round-mode=places, round-integer-to-decimal, round-precision=4,
    table-format = 1.4, 
    table-number-alignment=center,
    round-integer-to-decimal]{siunitx}
\usepackage{adjustbox}  

\newcommand{\Ak}{T^k_{\theta}}
\newcommand{\Aki}{T^k_{\theta_i}}
\newcommand{\op}{T_{\theta}}

\newcommand{\Lo}{\ell_{\theta}}
\newcommand{\Loi}{\ell_{\theta_i}}
\newcommand{\hw}{h_{\mathcal{W}}}

\newcommand{\zwh}{\hat{z}_{\mathcal{W}}(\theta)}
\newcommand{\pfc}{2 \beta^k}
\newcommand{\pfct}{2 \sqrt{2} \beta^k}
\newcommand{\pfcf}{4 \sqrt{2} \beta^k}

\newcommand{\fath}{\forall \theta \in \Theta}

\DeclareMathOperator*{\argmax}{arg\,max}

\newcommand{\ie}{{\it i.e.}}

\newcommand{\reals}{{\mbox{\bf R}}}

\newcommand{\symm}{{\mbox{\bf S}}}  

\newcommand{\Expect}{\mathop{\bf E{}}}

\newcommand{\dist}{{\bf dist{}}}
\newcommand{\fix}{\mathop{\bf fix{}}}
\newcommand{\vect}{\mathop{\bf vec{}}}

\newcommand{\Rad}{{\bf rad}}
\newcommand{\ERad}{{\bf erad}}

\newcommand{\Sec}{Sec.\;}
\newcommand{\Thm}{Thm.\;}
\newcommand{\Page}{pp\;}
\newcommand{\Exmp}{Ex.\;}
\newcommand{\Eqn}{Eq.\;}
\newcommand{\Cor}{Cor.\;}
\newcommand{\Prop}{Prop.\;}

\begin{document}

\maketitle

\begin{abstract}
First-order methods are widely used to solve convex quadratic programs (QPs) in real-time applications because of their low per-iteration cost. 
However, they can suffer from slow convergence to accurate solutions. 
In this paper, we present a framework which learns an effective warm-start for a popular first-order method in real-time applications, Douglas-Rachford (DR) splitting, across a family of parametric QPs.
This framework consists of two modules: a feedforward neural network block, which takes as input the parameters of the QP and outputs a warm-start, and a block which performs a fixed number of iterations of DR splitting from this warm-start and outputs a candidate solution. 
A key feature of our framework is its ability to do end-to-end learning as we differentiate through the DR iterations. 
To illustrate the effectiveness of our method, we provide generalization bounds (based on Rademacher complexity) that improve with the number of training problems and number of iterations simultaneously. 
We further apply our method to three real-time applications  and observe that, by learning good warm-starts, we are able to significantly reduce the number of iterations required to obtain high-quality solutions.
\end{abstract}

\ifpreprint \else
\fi



\section{Introduction}
We consider the problem of solving convex quadratic programs (QPs) within strict real-time computational constraints using first-order methods. 
QPs arise in various real-time applications in robotics~\citep{tedrake_robotics_qp}, control~\citep{borrelli_mpc_book}, and finance~\citep{cvx_portfolio}. 
In the past decade, first-order methods have gained a wide popularity in real-time quadratic optimization~\citep{Boyd_admm, fom_book, lscomo} because of their low per-iteration cost and their warm-starting capabilities. 
However, they still suffer from slow convergence to the optimal solutions, especially for badly-scaled problems~\citep{fom_book}. 
As a workaround to this issue, one can make use of the oftentimes parametric nature of the QPs which feature in real-time applications.  
For example, one can use the solution to a previously solved QP as a warm-start to a new problem~\citep{qp_oases,osqp}.
While this approach is popular, it only makes use of the data from the previous problem, neglecting the vast majority of data available.
More recent approaches in machine learning have sought to exploit data by solving many different parametric problems offline to learn a direct mapping from the parameters to the optimal solutions. 
The learned solution is then used as a warm-start~\citep{mpc_primal_active_sets,warm_start_power_flow}.
These approaches require solving many optimization problems to optimality, which can be expensive, and they also do not take into consideration the characteristics of the algorithm that will run on this warm-start downstream. 
Furthermore, such learning schemes often do not provide generalization guarantees~\citep{amos_tutorial} on the algorithmic performance on unseen data.
Such guarantees are crucial for real-time and safety critical applications where the algorithms must return high-quality solutions within strict time limits.

\paragraph{Contributions.} In this work, we exploit data to learn a mapping from the parameters of the QP to a warm-start of a popular first-order method, Douglas-Rachford (DR) splitting. 
The goal is to decrease the number of real-time iterations of DR splitting that are required to obtain a good-quality solution in real-time.
Our contributions are the following:
\begin{figure}[t!]
\centering
\ifpreprint
\includegraphics[width=0.9\columnwidth]{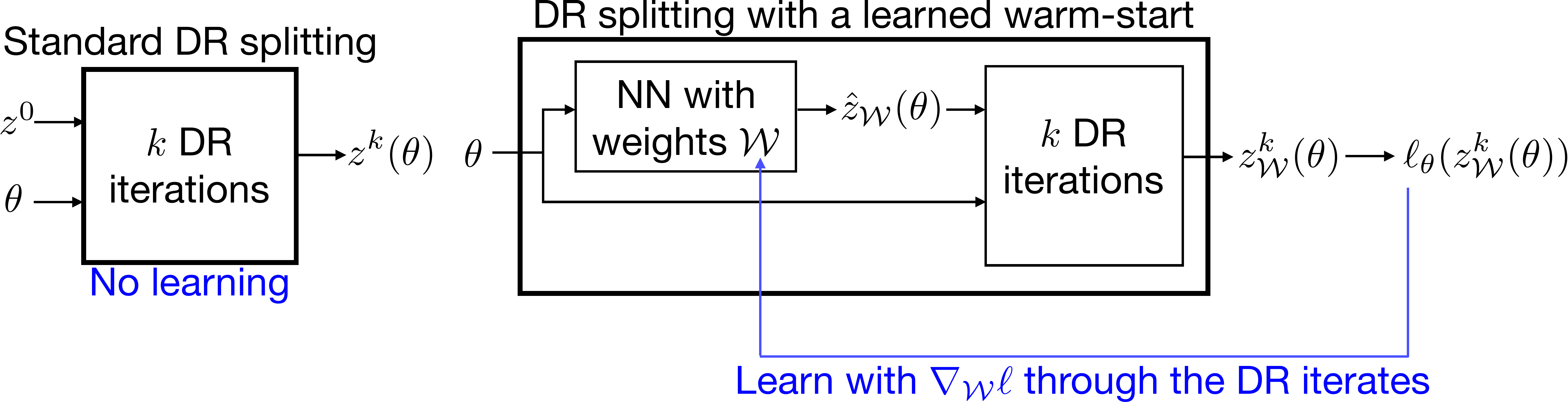}
\else
\includegraphics[width=.8\columnwidth]{figures/learning_framework_diagram-crop.pdf}
\fi
\caption{Left: standard DR splitting which maps parameter $\theta$ and initialization $z^0$ to an approximate solution $z^k(\theta)$. 
Right: Proposed learning framework consisting of two modules.
The first module is the NN block which maps the parameter $\theta$ to a warm-start $\zwh$. 
The weights of the NN, denoted by $\mathcal{W}$, are the only variables we optimize over. 
The second module runs $k$ iterations of DR splitting (which also depend on $\theta$) starting with the warm-start $\zwh$ and returning a candidate solution~$z^k_{\mathcal{W}}(\theta)$. 
We backpropagate from the loss $\Lo(z^k_{\mathcal{W}}(\theta))$ through the DR iterates to learn the optimal weights $\mathcal{W}$.
}
\ifpreprint \else
\vspace{-1.5em}
\fi
\label{fig:nn_architecture}
\end{figure}
\begin{itemize}
    \item We propose a principled framework to learn high quality warm-starts from data. This framework consists of two modules as indicated in Figure \ref{fig:nn_architecture}. The first module is a feedforward neural network (NN) that predicts a warm-start from the problem parameters. The second module consists of $k$ DR splitting iterations that output the candidate solution. We differentiate the loss function with respect to the neural network weights by backpropagating through the DR iterates, which makes our framework an end-to-end warm-start learning scheme. Furthermore, our approach does not require us to solve optimization problems offline. 
    \item We combine operator theory and Rademacher complexity theory to obtain novel generalization bounds that guarantee good performance for parametric QPs with unseen data. The bounds improve with the number of training problems and the number of DR iterations simultaneously, thereby allowing great flexibility in our learning task. 
    \item We benchmark our approach on real-time quadratic optimization examples, showing that our method can produce an excellent warm-start that reduces the number of DR iterations required to reach a desired accuracy by at least $30 \%$ and as much as $90\%$ in some cases.
\end{itemize}

\section{Related work}\label{section:related_work}

\paragraph{Learning warm-starts.}
A common approach to reduce the number of iterations of iterative algorithms is to learn a mapping from problem parameters to high-quality initializations. 
In optimal power flow, \cite{warm_start_power_flow} trains a random forest to predict a warm-start.
In the model predictive control (MPC)~\citep{borrelli_mpc_book},~\cite{mpc_primal_active_sets} use a neural network to accelerate the optimal control law computation by warm-starting an active set method.
Other works in MPC use machine learning to predict an approximate optimal solution and, instead of using it to warm-start an algorithm, they directly ensure feasibility and optimality.
\cite{mpc_constrained_neural_nets} and \cite{deep_learning_mpc_karg} use a constrained neural network architecture that guarantees feasibility by projecting its output onto the QP feasible region.
\cite{borrelli_mpc_primal_dual} uses a neural network to predict the solution while also certifying suboptimality of the output.
In these works, the machine learning models do not consider that additional algorithmic steps will be performed after warm-starting.
Our work differs in that the training of the NN is designed to minimize the loss after many steps of DR splitting.
Additionally, our work is more general in scope since we consider general parametric QPs.

\paragraph{Learning algorithm steps.}
There has been a wide array of works to speedup machine learning tasks by tuning algorithmic steps of stochastic gradient descent methods~\citep{learning_to_optimize_malik, learn_learn_gd_gd,VeLO,l2o,amos_tutorial}. 
Similarly,~\cite{lista} and~\cite{alista} accelerate the solution of sparse encoding problems by learning the step-size of the iterative soft thresholding algorithm. 
Operator splitting algorithms~\citep{lscomo} can also be speed up by learning acceleration steps~\citep{neural_fp_accel_amos} or the closest contractive fixed-point iteration to achieve fast convergence~\citep{closest_fp_op}.
Reinforcement learning has also gained popularity as a versatile technique to accelerate the solution of parametric QPs by learning a policy to tune the step size of first-order methods~\citep{rlqp,qp_accelerate_rho}.
A common tactic in these works is to differentiate through the steps of an algorithm to minimize a performance loss using gradient-based method. This known as {\it loop unrolling} which has been used in other areas such as meta-learning~\citep{maml} and variational autoencoders~\citep{semi_amortiized_vae}.
While we also unroll the algorithm iterations, our works differs in that we learn a high-quality warm-start rather than the algorithm steps.
This allows us to guarantee convergence and also provide generalization bounds over the number of iterations.

\paragraph{Learning surrogates.}
Instead of solving the original parametric problem, several works aim to learn a surrogate model that can be solved quickly in real-time applications.
For example, by predicting which constraints are active~\citep{learn_active_sets} and the value of the optimal integer solutions~\citep{voice_optimization,online_milliseconds} we can significantly accelerate the real-time solution of mixed-integer convex programs by solving, instead, a surrogate low-dimensional convex problem.
Other approaches lean a mapping to reduce the dimensionality of the decision variables in the surrogate problem~\citep{surrogate_learning}.
This is achieved by embedding such problem as an implicit layer of a neural network and differentiating its KKT optimality conditions~\citep{diff_opt,amos2018differentiable, diffcp2019}.
In contrast, our method does not approximate any problem and, instead, we predict a warm-start of the algorithmic procedure with a focus on real-time computations. 
This allows us to clearly quantify the suboptimality achieved within a fixed number of real-time iterations.

\section{End-to-end learning framework} \label{sec:framework}
\paragraph{Problem formulation.}
We consider the following parametric (convex) QP
\begin{equation}\label{prob:primal_dual_prob}
\begin{array}{ll} \mbox{minimize} & (1/2) x^T P x + c^T x \\
\mbox{subject to} & Ax + s = b \\
& s \geq 0
\end{array}
\quad \text{ with parameter } \quad \theta = (\vect(P), \vect(A), c, b) \in \reals^d,
\end{equation}
and decision variables $x \in \mathbf{R}^n$ and $s \in \mathbf{R}^m$. 
Here, $P$ is a positive semidefinite matrix in $\symm_+^{n \times n}$, $A$ is a matrix in $\reals^{m \times n}$, and $b$ and $c$ are vectors in $\reals^m$ and $\reals^n$ respectively. 
For a matrix $Y$, $\vect(Y)$ denotes the vector obtained by stacking the columns of $Y$. 
The dimension $d$ of $\theta$ is upper bounded by $mn+n^2+m+n$, but can be smaller in the case where only some of the data changes across the problems. 
Our goal is to quickly solve the QP in~\eqref{prob:primal_dual_prob} with $\theta$ randomly drawn from a distribution $\mathcal{D}$ with compact support set $\Theta$, assuming that it admits an optimal solution for any $\theta \in \Theta$.

\paragraph{Optimality conditions.}
The KKT optimality conditions of problem~\eqref{prob:primal_dual_prob}, that is, primal feasibility, dual feasibility, and complementary slackness, are given by
\begin{equation}\label{eq:KKT}
Ax + s = b, \quad A^T y + Px + c = 0, \quad s \geq 0, \quad y \geq 0, \quad s \perp y = 0,
\end{equation}
where $y \in \mathbf{R}^m$ is the dual variable to problem~\eqref{prob:primal_dual_prob}.
We can compactly write these conditions as a linear complementarity problem~\citep[\Sec 3]{scs_quadratic}, \ie, the problem of finding a $u=(x,y) \in \reals^{m + n}$ such that 
\begin{equation}\label{prob:LCP_problem_compact}
\mathcal{C} \ni u \perp Mu + q \in \mathcal{C}^*,
  \quad \text{where}
  \quad M = \begin{bmatrix}
      P & A^T \\
      -A & 0
      \end{bmatrix}\in \reals^{(m+n) \times (m+n)}, 
\end{equation}
and $q=(c, b) \in \reals^{m+n}$.  
Here, $\mathcal{C} = \reals^n \times \reals_+^m$ and $\mathcal{C}^* = \{0\}^n \times \mathbf{R}_+^m$ denotes the dual cone to $\mathcal{C}$, \ie, $\mathcal{C}^* = \{w \mid w^Tu \ge 0,\;u \in \mathcal{C}\}$. 
This problem is equivalent to finding $u \in \reals^{m +m}$ that satisfies the following inclusion~\citep[\Exmp 26.22]{bauschke2011convex}\citep[\Sec 3]{scs_quadratic}
\begin{equation}\label{prob:monotone_inclusion_prob}
0 \in Mu + q + N_{\mathcal{C}}(u),
\end{equation}
where $N_{\mathcal{C}}(u)$ is the normal cone for cone $\mathcal{C}$ defined as $N_{\mathcal{C}}(z) = \{x \mid (y - u)^Tx \le 0,\forall y \in \mathcal{C}\}$ if $u \in \mathcal{C}$ and $\emptyset$ otherwise. 
Of importance to us to ensure convergence of the algorithm we define next is the fact that $Mu+q+N_{\mathcal{C}}(u)$ is maximal monotone; see~\citep[\Sec 2.2]{lscomo} for a definition. 
This follows from $P \succeq 0$, $\mathcal{C}$ being a convex polyhedron, and~\eqref{prob:primal_dual_prob} always admitting an optimal solution~\citep[\Thm 7]{lscomo}~\citep[\Thm 11]{lscomo}.

\begin{algorithm}[h]
  \caption{The DR Splitting algorithm for $k$ iterations to solve problem~\eqref{prob:monotone_inclusion_prob}.
  }\label{alg:DR-splitting}
    \textbf{Inputs:} initial point $z^0$, problem data $(M, q)$, tolerance $\epsilon$, $k$ number of iterations\\
    \textbf{Output: } approximate solution $z^k$ \\
\For{$i=0,\dots,k-1$}{
	$u^{i+1} = \left(M+I\right)^{-1}\left(z^i - q\right)$\\
	$\tilde{u}^{i+1} = \Pi_{\mathcal{C}}\left(2u^{i+1} - z^i\right)$\\
	$z^{i+1} = z^i + \tilde{u}^{i+1} - u^{i+1}$
}
\end{algorithm}
\paragraph{Douglas-Rachford splitting.}
We apply Douglas-Rachford (DR) splitting~ \citep{dr_splitting_sum_operators,dr_splitting} to solve problem~\eqref{prob:monotone_inclusion_prob}.
DR consists of evaluating the {\it resolvent} of operators $Mu + q$ and $N_{\mathcal{C}}$, which for an operator $F$ is defined as $(I + F)^{-1}$~\citep[\Page 40]{lscomo}.
By noting that the resolvent of $Mu + q$ is $(M+I)^{-1}(z-q)$ and the resolvent of $N_{\mathcal{C}}(u)$ is 
$\Pi_{\mathcal{C}}(z)$, \ie, the projection onto $\mathcal{C}$~\citep[\Eqn 2.8, \Page 42]{lscomo}, we obtain Algorithm~\ref{alg:DR-splitting}.

The linear system in the first step is always solvable since $M + I$ has full rank~\citep{scs_quadratic}, but it varies from problem to problem. 
The projection onto $\mathcal{C}$, however, is the same for all problems and simply clips negative values to zero and leaves non-negative values unchanged. 
For compactness, in the remainder of the paper, we write Algorithm~\ref{alg:DR-splitting} as
\begin{equation}\label{eq:op.DR}
z^{i+1} = \op \left(z^i\right)\quad \text{where} \quad \op(z) = z + \Pi_{\mathcal{C}} \left( 2 (M+I)^{-1}(z-q) - z \right) - (M+I)^{-1}(z-q).
\end{equation}
We make the dependence of $T$ on $\theta$ explicit here as $M$ and $q$ are parametrized by $\theta$. 
DR splitting is guaranteed to converge to a fixed point $z^\star \in \fix \op$ such that $\op(z^\star) = z^\star$. Algorithm~\ref{alg:DR-splitting} returns an approximate solution $z^k$, from which we can recover an approximate primal-dual solution to~\eqref{prob:primal_dual_prob} by computing $(x^k, y^k) = u^k = (M+I)^{-1}(z^k - q)$ and $s^k = b - Ax^k$.

\paragraph{Our end-to-end learning architecture.} Our architecture consists of two modules as in Figure~\ref{fig:nn_architecture}. The first module is a NN with weights $\mathcal{W}$: it predicts a good-quality initial point (or warm-start), $\zwh$, to Algorithm~\ref{alg:DR-splitting}  from the parameter $\theta$ of the QP in \eqref{prob:primal_dual_prob}. We assume that the NN has $L$ layers with ReLU activation functions~\citep{relu_popular}. We then write
\begin{equation*}
  \zwh = \hw(\theta) = h_L\left(h_{L-1} \dots h_1(\theta)\right),
\end{equation*}
where $h_l(y_l) = \left(W_l y_l + b_l\right)_+$, with $\{W_l\}_{l=0}^L$ being the weight matrices and $\{b_l\}_{l=0}^L$ the bias terms. Here, $\hw(\cdot)$ is a mapping from $\reals^{d}$ to $\reals^{m+n}$ corresponding to the prediction and we denote the set of all such mappings by $\mathcal{H}$. We emphasize the dependency of $h$ on the weights and bias terms via the subscript $\mathcal{W} = (W_1, b_1, \dots, W_L, b_L)$. 
The second module corresponds to $k$ iterations of DR splitting from the initial point $\zwh$. It outputs an approximate solution $z^k_{\mathcal{W}}(\theta)$, from which we can recover an approximate solution to~\eqref{prob:primal_dual_prob} as explained above. Using the operator in~\eqref{eq:op.DR}, we write
\begin{equation*}
\Ak(\zwh)=z^k_{\mathcal{W}}(\theta).
\end{equation*}
To obtain the solution to a QP given parameter $\theta$, we simply need to perform a forward pass of the architecture, \ie,  compute $\Ak(\hw(\theta))$, with $k$ chosen as needed. 



\paragraph{Learning task.}
We define the loss function as the fixed-point residual of operator $T_{\theta}$, \ie,
\begin{equation}
\label{eq:loss}
\Lo(z) = \|\op(z) - z\|_2.
\end{equation}
This loss measures the distance to convergence of Algorithm~\ref{alg:DR-splitting}. 
The goal is to minimize the expected loss, which we define as the {\it risk},
\begin{equation*}
R(\hw) = \mathbf{E}_{\theta \sim \mathcal{D}}\left[\Lo\left(\Ak(\hw(\theta))\right)\right],
\end{equation*}
with respect to the weights $\mathcal{W}$ of the NN.
In general, we cannot evaluate $R(\hw)$ exactly and, instead, 
we  minimize the {\it empirical risk}
\begin{equation}
\hat{R}(\hw) = (1/N) \sum_{i=1}^N \Loi \left(\Aki(\hw(\theta_i))\right).
\label{eq:empirical-risk}
\end{equation}
Here, $N$ is the number of training problems and we work with a full-batch approximation, though mini-batch or stochastic approximations can also be used~\citep{opt_for_ml}.



\paragraph{Differentiability of our architecture.} To see that we can differentiate $\Lo$ with respect to $\mathcal{W}$, note that the second module consists of repeated linear system solves and projections onto $\mathcal{C}$ (see Algorithm~\ref{alg:DR-splitting}). Since the linear systems always have unique solutions, $u^{i+1}$ is linear in $z^i$ and the linear system solves are differentiable. Furthermore, as the projection step involves clipping non-negative values to zero, it is differentiable everywhere except at zero. In the first module, the NN consists entirely of differentiable functions except for the ReLU activation function, which is likewise differentiable everywhere except at zero. 

\section{Generalization bounds}\label{section:gen}

In this section, we provide an upper bound on the expected loss $R(h_{\mathcal{W}})$ of our framework for any $h_{\mathcal{W}} \in \mathcal{H}$. This bound involves the empirical expected loss $\hat{R}(h_{\mathcal{W}})$, the Rademacher complexity of the NN appearing in the first module \emph{only} and a term which decreases with both the number $k$ of iterations of DR splitting and the number $N$ of training samples. To obtain this bound, we rely on the fact that DR splitting on~\eqref{prob:monotone_inclusion_prob} achieves a linear convergence rate. More specifically, following \citep[\Thm 1]{banjac_scs_lin_conv}, we have that
\begin{equation}\label{eq:lin_conv}
    \dist_{\fix \op}(\op(z)) \leq \beta_{\theta} \dist_{\fix \op}(z)
\end{equation}
where $\dist_{\mathcal{S}}(x)=\min \{||x-y|| ~|~ y\in \mathcal{S}\}$ and $\beta_{\theta} \in (0,1)$ is the rate of linear convergence for problem with parameter $\theta$. We now state the result.

\begin{theorem}\label{thm:rademacher}
Let $\beta=\max_{\theta \in \Theta} \beta_{\theta}$ for $\beta_{\theta}$ as in~\eqref{eq:lin_conv}.
Assume that $\mathcal{H}$ is the set of mappings defined in Section~\ref{sec:framework} with the additional assumption that for any $h_{\mathcal{W}} \in \mathcal{H}$, $\dist_{\fix \op}(\hw(\theta)) \leq B$ for some $B > 0$ and any $\theta \in \Theta$. Then, 
with probability at least $1 - \delta$ over the draw of i.i.d samples,
\begin{equation*}
R(h_{\mathcal{W}}) \leq \hat{R}(h_{\mathcal{W}}) + 2 \sqrt{2} \beta^k \left(2 \Rad(\mathcal{H}) + B \log(1/\delta)/(2N)\right),\quad \forall h_{\mathcal{W}} \in \mathcal{H},
\end{equation*}
where $k$ is the number of iterations of DR splitting in the second module, $N$ is the number of training samples, $\Rad(\mathcal{H})$ is the Rademacher complexity of $\mathcal{H}$, and $\beta \in (0,1)$.
\end{theorem}
In settings where we can upper bound the Rademacher complexity of $\mathcal{H}$, for example in the case of NNs which are linear functions with bounded norm, or 2-layer NNs with ReLU activation functions, we are able to provide a bound on the generalization error of our framework which makes the dependence on $k$ and $N$ even more explicit~\citep{2layerrademacher,two_layer_relu_gen}. 

\begin{corollary}
Let $\mathcal{H}$ be the set of linear functions with bounded norm, i.e., $\mathcal{H}=\{h \mid h(\theta)=W\theta\}$ where $\theta \in \reals^d$, $W \in \reals^{(m+n) \times d}$ and $(1/2)||W||_F^2 \leq B$ for some $B > 0$. Then, with probability at least $1 - \delta$ over the draw of i.i.d samples,  
\begin{equation*}
R(h_{\mathcal{W}}) \leq \hat{R}(h_{\mathcal{W}}) + 2\sqrt{2}\beta^k \left(2 \rho_2(\theta) \sqrt{2d/N} + B \log(1/\delta)/(2N)\right),\quad \forall h_{\mathcal{W}} \in \mathcal{H},
\end{equation*}
where $k$ and $N$ are as defined in Theorem \ref{thm:rademacher}, and $\rho_2(\theta)=\max_{\theta \in \Theta}||\theta||_2$~\citep[\Thm 5.10]{ml_foundations}.
\end{corollary}

\section{Numerical experiments}
We now illustrate our method with examples of quadratic optimization problems deployed and repeatedly solved in control and portfolio optimization settings where rapid solutions are important for real-time execution and backtesting.
Our architecture was implemented in the JAX library~\citep{jax} with Adam~\citep{adam} training optimizer.
All computations were run on the Princeton HPC Della Cluster.
We use $10000$ training problems and evaluate on $2000$ test problems.
In our examples we use a NN with three hidden layers of size $500$ each. 
Our code is available at \url{https://github.com/stellatogrp/l2ws}.

\newcommand{\cblock}[3]{
  \hspace{-1.5mm}
  \begin{tikzpicture}
    [
    node/.style={square, minimum size=10mm, thick, line width=0pt},
    ]
    \node[fill={rgb,255:red,#1;green,#2;blue,#3}] () [] {};
  \end{tikzpicture}%
}
\begin{figure}
    \centering
    \ifpreprint
    \includegraphics[width=\textwidth]{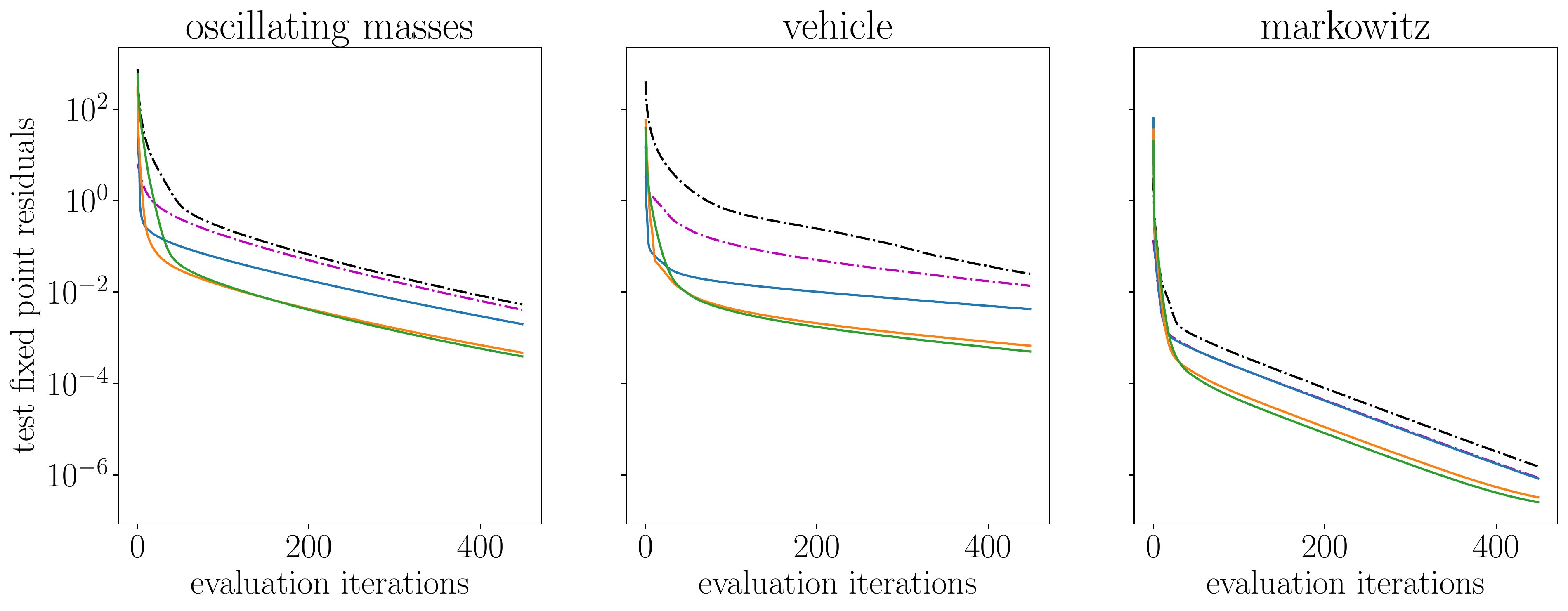}
    \else
    \includegraphics[width=.8\textwidth]{figures/all_eval_iters.pdf}
    \fi
    \cblock{0}{0}{0} no warm-start \hspace{1mm}
    \cblock{191}{0}{191} nearest neighbor warm-start \hspace{1mm}
    learned warm-start $k=$\{\hspace{-1mm}\cblock{31}{119}{180} $5$
        \cblock{255}{127}{14} $15$
        \cblock{44}{160}{44} $50$\}\hspace*{-2mm}
    \caption{
    We plot the test fixed point residuals for different warm-starts of DR splitting.
    We train our architecture with $k=5,15,$ and $50$ DR iterations with loss function \eqref{eq:empirical-risk}.
    We compare our results against a random initialization (black) and against warm-starting DR splitting with the nearest neighbor from the train set (magenta).
    Left: oscillating masses example.
    Middle: vehicle dynamics example.
    Right: portfolio optimization example.
    }
    \label{fig:result_plots}
\end{figure}

\subsection{Oscillating masses}
We consider the problem of controlling a physical system that involves connected springs and masses~\citep{fast_mpc_boyd},\citep[System 4]{mpc_primal_active_sets}. This can be formulated as the following QP:
\begin{equation*}
\begin{array}{ll} 
\label{eq:oscmasses}
\mbox{minimize} & x_{T}^T Q_{T} x_{T} + \sum_{t=1}^{T-1} x_{t}^T Q x_{t} + u_t^T R u_t \\
\mbox{subject to} & x_{t+1} = A x_t + B u_t \quad t=0, \dots, T-1, \\
& u_{\textrm{min}} \leq u_t \leq u_{\textrm{max}} \quad t=0, \dots, T-1 \\
& x_{\textrm{min}} \leq x_t \leq x_{\textrm{max}} \quad t=1, \dots, T,\\
& x_0 = x_{\rm init}\\
\end{array}
\end{equation*}
where the states $x_t \in \reals^n_x$ and the inputs $u_t \in \reals^{n_u}$ are subject to lower and upper bounds. Matrices $A \in \reals^{n_x \times n_x}$ and $B\in\reals^{n_x \times n_u}$ define the system dynamics. 
The horizon length is $T$ and the parameter $\theta$ is initial state $x_{\rm init}$. Matrices $Q\in \symm_{+}^{n_x}$ and $R \in \symm_{++}^{n_u}$ define the state and input costs at each stage, and $Q_T \in \symm_{+}^{n_x}$ the final stage cost.

\paragraph{Numerical example.}
We consider $n_x=36$ states, $n_u=9$ control inputs, and a time horizon of $T=50$. Matrices $A$ and $B$ are obtained by discretizing the state dynamics with time step $0.5$~\citep[System 4]{mpc_primal_active_sets}.
We set $u_{\textrm{min}} = -u_{\textrm{max}} = 1/2$ and $x_{\textrm{min}} = -x_{\textrm{max}} = 2$.
We set $Q_t=I$ for $t=1, \dots, T$, and $R_t=I$ for $t=1, \dots, T-1$.
We sample $\theta=x_{\rm init}$ uniformly in~$[-2,2]^{36}$. 
Figure~\ref{fig:result_plots} and Table~\ref{tab:oscmasses} show the convergence behavior of our method.
\begin{table}
\ifpreprint 
\small
\else
  \footnotesize
  \fi
  \centering
    \renewcommand*{\arraystretch}{1.0}
  \caption{Oscillating masses problem. 
  We compare the number of iterations of DR splitting required to reach different levels of accuracy with different warm-starts (learned warm-start with $k=5,15,50$, no warm-start, and a nearest neighbor warm-start).
  The reduction columns are the iterations reduced as a fraction of the no learning iterations.} 
  \label{tab:oscmasses}
  \begin{tabular}{llllllllll}
    \toprule
    ~ & no learning & \multicolumn{2}{c}{nearest neighbor} & \multicolumn{2}{c}{train $k=5$} & \multicolumn{2}{c}{train $k=15$} & \multicolumn{2}{c}{train $k=50$}\\
     \cmidrule(lr){2-2} \cmidrule(lr){3-4} \cmidrule(lr){5-6} \cmidrule(lr){7-8} \cmidrule(lr){9-10}
    $\epsilon$ & iters & iters & reduction & iters & reduction & iters & reduction & iters & reduction \\
    \midrule
    \begin{filecontents*}{./data/mpc_table.csv}
,accuracies,no_learn_iters,naive_ws_iters,naive_ws_red,traink5_iters,traink5_red,traink15_iters,traink15_red,traink50_iters,traink50_red
0,381,0.01,353,0.07,279,0.27,176,0.54,127,0.67
1,651,0.001,616,0.05,555,0.15,438,0.33,338,0.48
2,1019,0.0001,973,0.05,932,0.09,816,0.20,663,0.35
\end{filecontents*}
\csvreader[head to column names, late after line=\\]{./data/mpc_table.csv}{
    accuracies=\colB,
    no_learn_iters=\colA,
    naive_ws_iters=\colC,
    naive_ws_red=\colD,
    traink5_iters=\colE,
    traink5_red=\colF,
    traink15_iters=\colG,
    traink15_red=\colH,
    traink50_iters=\colI,
    traink50_red=\colJ,
    }{\colA & \colB & \colC & \colD & \colE & \colF & \colG & \colH & \colI & \colJ} 
    \bottomrule
  \end{tabular}
\end{table}


\subsection{Vehicle dynamics control problem}
We consider problem of controlling a vehicle, modeled as a parameter-varying linear dynamical system~\citep{vehicle_dynamics}, to track a reference trajectory~\citep{borrelli_mpc_primal_dual}. We formulate it as the following QP
\begin{equation*}
\begin{array}{ll} 
\mbox{minimize} & (y_T - y_T^{\textrm{ref}})^T Q (y_T - y_T^{\textrm{ref}}) + \sum_{t=1}^{T-1} (y_t - y_t^{\textrm{ref}})^T Q_T (y_t - y_t^{\textrm{ref}}) + u_t^T R u_t \\
\mbox{subject to} & x_{t+1} = A(v) x_t + B(v) u_t + E(v) \delta_t \quad t=0, \dots, T-1 \\
&|u_t| \leq  \bar{u}, \quad | u_t - u_{t-1}| \leq  \overline{\Delta u}, \quad t=0, \dots, T-1 \\
& y_t = C x_t,\quad t=0,\dots,T-1\\
& x_0 = x_{\rm init},
\end{array}
\end{equation*}
where $x_t \in \reals^4$ is the state and $u_t \in \reals^3$ is the input, and $\delta_t \in \reals$ is the driver steering input, which we assume is linear over time.
We aim to minimize the distance between the output, $y_t \in \reals^3$ and the reference trajectory $y_t^{\textrm{ref}} \in \reals^3$ over time.
Matrices $Q\in \symm_{+}^{n_x}$ and $Q_T \in \symm_{+}^{n_x}$ define the state costs, $R \in \symm_{++}^{n_u}$ the input cost, and $C \in \reals^{3 \times 4}$ the output $y_t$.
The term $v \in \reals$ is the longitudinal velocity of the vehicle that parametrizes $A \in \reals^{4 \times 4}, B \in \reals^{4 \times 3}$, and $E \in \reals^{4}$.
Vectors $\bar{u}$ and $\overline{\Delta u}$ bound the magnitude of the inputs and change in inputs respectively. 
The parameters $\theta$ for the problem are the initial state $x_{\rm init}$, the initial velocity~$v$, the previous control input~$u_{-1}$, the reference signals~$y_t^{\rm ref}$ for $t=0,\dots,T$, and the steering inputs~$\delta_t$ for $t=0,\dots,T-1$.

\paragraph{Numerical example.} 
The time horizon is $T=30$.
Matrices $A,B,E$ result from discretizing the dynamics~\citep{vehicle_dynamics}.
We sample all parameters uniformly from their bounds: the velocity $v \in [2,35]$, the output $y_t \in [-\bar{y}, \bar{y}]$ where $\bar{y}=(25, 40, 30)$ in degrees, and the previous control $u_{-1} \in [-\bar{u}, \bar{u}]$ where $\bar{u} = 10^3 (1, 20, 30)$.
We sample the initial steering angle from $[-45,45]$ and its linear increments from $[-30,30]$.
Figure~\ref{fig:result_plots} and Table~\ref{tab:vehicle} show the performance of our method.
\begin{table}
\ifpreprint \small \else \footnotesize \fi
  \centering
      \renewcommand*{\arraystretch}{1.0}
  \caption{Vehicle problem. 
  We compare the number of iterations of DR splitting required to reach different levels of accuracy with different warm-starts (learned warm-start with $k=5,15,50$, no warm-start, and a nearest neighbor warm-start).
  The reduction columns are the iterations reduced as a fraction of the no learning iterations.} 
  \label{tab:vehicle}
  \begin{tabular}{llllllllll}
    \toprule
    ~ & no learning & \multicolumn{2}{c}{nearest neighbor} & \multicolumn{2}{c}{train $k=5$} & \multicolumn{2}{c}{train $k=15$} & \multicolumn{2}{c}{train $k=50$}\\
    \cmidrule(lr){2-2} \cmidrule(lr){3-4} \cmidrule(lr){5-6} \cmidrule(lr){7-8} \cmidrule(lr){9-10}
    $\epsilon$ & iters & iters & reduction & iters & reduction & iters & reduction & iters & reduction \\
    \midrule
    \begin{filecontents*}{./data/vehicle_table.csv}
,accuracies,no_learn_iters,naive_ws_iters,naive_ws_red,traink5_iters,traink5_red,traink15_iters,traink15_red,traink50_iters,traink50_red
0,639,0.01,520,0.19,203,0.68,48,0.92,48,0.92
1,1348,0.001,1163,0.14,895,0.34,351,0.74,299,0.78
2,2126,0.0001,1948,0.08,1653,0.22,1006,0.53,882,0.59
\end{filecontents*}
\csvreader[head to column names, late after line=\\]{./data/vehicle_table.csv}{
    accuracies=\colB,
    no_learn_iters=\colA,
    naive_ws_iters=\colC,
    naive_ws_red=\colD,
    traink5_iters=\colE,
    traink5_red=\colF,
    traink15_iters=\colG,
    traink15_red=\colH,
    traink50_iters=\colI,
    traink50_red=\colJ,
    }{\colA & \colB & \colC & \colD & \colE & \colF & \colG & \colH & \colI & \colJ} 
    \bottomrule
  \end{tabular}
\end{table}


\subsection{Portfolio optimization}
We consider the portfolio optimization problem where we want to allocate assets to maximize the risk-adjusted return~\citep{markowitz,cvx_portfolio},
\begin{equation*}
\begin{array}{ll} \mbox{maximize} & \rho \mu^T x - x^T \Sigma x\\
\mbox{subject to} & \mathbf{1}^T x = 1,\quad x \geq 0,
\end{array}
\end{equation*}
where $x\in \reals^{n}$ represents the portfolio, $\mu \in \reals^n$ the expected returns, $1/\rho > 0$ the risk-aversion parameter, and $\Sigma \in \symm_{+}^{n}$ the return covariance.
For this problem, $\theta = \mu$.

\paragraph{Numerical example.}
We use real-world stock return data from $3000$ popular assets from 2015-2019 \citep{eod}.
We use an $l$-factor model for the risk and set $\Sigma = F \Sigma_F F^T + D$ where $F \in \reals^{n,l}$ is the factor-loading matrix, $\Sigma_F \in \symm_{+}^{l}$ estimates the factor returns, and $D \in \symm_{+}^{n}$ is a diagoal matrix accounting for additional variance for each asset also called the idiosyncratic risk.
We compute the factor model with $15$ factors by using the same approach as in \citep{cvx_portfolio}.
The return parameters are $\mu = \alpha (\hat{\mu_t} + \epsilon_t)$ where $\hat{\mu_t}$ is the realized return at time $t$, $\epsilon_t \sim \mathcal{N}(0,\sigma_{\epsilon}I)$, and $\alpha=0.24$ is selected to minimize the mean squared error $\mathbf{E}\|\mu_t - \hat{\mu}_t\|_2^2$~\citep{cvx_portfolio}.
We iterate and repeatedly cycle over the five year period to sample a $\mu$ vector for each of our problems.
Figure~\ref{fig:result_plots} and Table~\ref{tab:portfolio} show the performance of our method.

\begin{table}
\ifpreprint \small \else \footnotesize \fi
  \centering
      \renewcommand*{\arraystretch}{1.0}
  \caption{Markowitz problem. 
  We compare the number of iterations of DR splitting required to reach different levels of accuracy with different warm-starts (learned warm-start with $k=5,15,50$, no warm-start, and a nearest neighbor warm-start).
  The reduction columns are the iterations reduced as a fraction of the no learning iterations.} 
  \label{tab:portfolio}
  \begin{tabular}{llllllllll}
    \toprule
    ~ & no learning & \multicolumn{2}{c}{nearest neighbor} & \multicolumn{2}{c}{train $k=5$} & \multicolumn{2}{c}{train $k=15$} & \multicolumn{2}{c}{train $k=50$}\\
    \cmidrule(lr){2-2} \cmidrule(lr){3-4} \cmidrule(lr){5-6} \cmidrule(lr){7-8} \cmidrule(lr){9-10}
    $\epsilon$ & iters & iters & reduction & iters & reduction & iters & reduction & iters & reduction \\
    \midrule
    \begin{filecontents*}{./data/markowitz_table.csv}
,accuracies,no_learn_iters,naive_ws_iters,naive_ws_red,traink5_iters,traink5_red,traink15_iters,traink15_red,traink50_iters,traink50_red
0,14,0.01,7,0.5,7,0.5,9,0.36,11,0.21
1,54,0.001,24,0.56,22,0.59,16,0.7,19,0.65
2,186,0.0001,148,0.2,147,0.21,72,0.61,61,0.67
\end{filecontents*}
\csvreader[head to column names, late after line=\\]{./data/markowitz_table.csv}{
    accuracies=\colB,
    no_learn_iters=\colA,
    naive_ws_iters=\colC,
    naive_ws_red=\colD,
    traink5_iters=\colE,
    traink5_red=\colF,
    traink15_iters=\colG,
    traink15_red=\colH,
    traink50_iters=\colI,
    traink50_red=\colJ,
    }{\colA & \colB & \colC & \colD & \colE & \colF & \colG & \colH & \colI & \colJ} 
    \bottomrule
  \end{tabular}
\end{table}


\appendix
\section{Proof of the generalization bound}


\subsection{Proof of Theorem \texorpdfstring{\ref{thm:rademacher}}{}}\label{proof:str_cvx_rademacher}
For ease of notation, let $s(\theta) = \Lo\left(\Ak(\hw(\theta))\right)$. 
The function class we consider is $\mathcal{S} = \{s \mid s(\theta) = \Lo\left(\Ak(\hw(\theta))\right), \hw \in \mathcal{H} \}$.
The risk and empirical risk for $s$ are  $R(s) = \mathbf{E}[s(\theta)]$ and $\hat{R}(s)= (1/N) \sum_{i=1}^N s(\theta_i)$, respectively.
Assume that the loss is bounded over all $s \in \mathcal{S}$ and $\theta \in \Theta$, $s(\theta) \leq b$. Note that $\hat{R}(s) = R(\hw)$ and $R(s) = \hat{R}(\hw)$. For any $\delta > 0$, with probability at least $1 - \delta$ over the draw of an i.i.d. sample $S$ of size $N$, each of the following holds ~\citep{bartlett_rademacher}\citep[\Thm 3.3]{ml_foundations}:
\begin{equation}\label{gen_bd_rad}
R(\hw) \leq \hat{R}(\hw) + 2 \Rad_S(\mathcal{S}) + b \sqrt{\log 1/\delta / (2N)},
\end{equation}
for all $\hw \in \mathcal{H}.$
To prove Theorem~\ref{thm:rademacher}, we need to bound $\Rad_S(\mathcal{S})$ and $b$.

\begin{lemma}\label{lemma:lemma1}
  Let $\op$ satisfy equation~\eqref{eq:lin_conv}. 
  Take $\beta=\max_{\theta \in \Theta} \beta_{\theta}$.
  Then, 
  \begin{equation}
    \Lo(\Ak(z)) \leq \pfc \dist_{\fix \op}(z) \quad \fath.
  \end{equation}
\end{lemma}
\begin{proof}
From the definition of the loss $\Lo$ in~\eqref{eq:loss}, for any $w \in \fix \op$ we have
  \begin{align*}
    \Lo(\Ak(z)) &= \|T^{k+1}_{\theta}(z)-\Ak(z)\|_2 \leq \|\Ak(z) - w\|_2 + \| T^{k+1}_{\theta}(z) - w\|_2 \\
    &\leq \|\Ak(z) - w\|_2 + \| \Ak(z) - w\|_2 \leq 2 \dist_{\fix \op}(\Ak(z)) \leq 2 \beta^k \dist_{\fix \op}(z),
  \end{align*}
where the first inequality comes from the triangle inequality, the second-inequality from $w \in \fix \op$ and non-expansiveness of $\op$, and the last inequality from the definition of $\dist_{\fix \op}$.
\end{proof}

\begin{lemma}\label{lemma:2}
Assume that $\forall \theta \in \Theta$, $\mathcal{H} \subseteq \{\hw \mid \dist_{\fix \op}(\hw(\theta)) \leq B\}$ and that $\op$ satisfies equation~\eqref{eq:lin_conv} with parameter $\beta$. 
Let $\epsilon$ and $\epsilon_j$ be i.i.d. Rademacher random variables. Then,
\begin{equation}
\textstyle
\Expect\left[\sup_{\hw \in \mathcal{H}} \epsilon \Lo \left(\Ak(\hw(\theta))\right)\right] \leq \pfct \mathbf{E}\left[\sup_{\hw \in \mathcal{H}}  \sum_j \epsilon_j \hw(z_j)\right].
\end{equation}
\end{lemma}

\begin{proof}
Denote $Z_{\theta}$ to be the set of possible predictions of $\mathcal{H}$ for a fixed $\theta$. That is, $Z_{\theta} = \{\hw(\theta) \mid \hw \in \mathcal{H}\}$. 
Let $\bar{z}_{\theta} = \argmax_{z \in Z_{\theta}} \Lo(\Ak(\theta)))$. 
\begin{align*}
\textstyle
2 \mathbf{E}\left[\sup_{\hw \in \mathcal{H}} \epsilon \Lo\left(\Ak(\hw(\theta))\right)\right] &= \textstyle 2 \mathbf{E}\left[\sup_{z \in Z_{\theta}} \epsilon \Lo\left(\Ak(z)\right)\right]\\
&\textstyle=\sup_{z,w \in Z_{\theta}} \Lo\left(\Ak(z)\right) - \Lo\left(\Ak(w)\right) = \Lo\left(\Ak(\bar{z}_{\theta})\right)\\
&\textstyle\leq \pfc \dist_{\fix \op}(\bar{z}_{\theta}) \leq \pfc \|\bar{z}_{\theta} - \tilde{z}_{\theta} \|_2\\
&\textstyle\leq \pfct \mathbf{E}\left[\left\lvert\sum_{j} \epsilon_j (\bar{z}_{\theta} - \tilde{z}_{\theta})_j\right\rvert \right] \\
&\textstyle\leq \pfct \mathbf{E}\left[\left\lvert\sup_{z,w \in Z_{\theta}} \sum_{k} \epsilon_j z_j - \sum_{j} \epsilon_j w_j \right\rvert \right]\\
&\textstyle= \pfct \left(\mathbf{E}\left[\sup_{z \in Z_{\theta}} \sum_j \epsilon_j z_j\right] + \mathbf{E}\left[\sup_{w \in \hw(\theta)} \sum_j -\epsilon_j w_j\right]\right)\\
&\textstyle= \pfcf \mathbf{E}\left[\sup_{\hw \in \mathcal{H}}  \sum_j \epsilon_j \hw(z_j)\right]
\end{align*}
The second line uses the definition of the Rademacher random variable. The supremum of the difference is achieved by maximizing $\Lo\left(\Ak(z)\right)$ for $z$ and picking $w \in \fix \op$. 
The third line uses Lemma~\ref{lemma:lemma1} and we pick any fixed point, $\tilde{z}_{\theta} \in \fix_{\op}$.
The fourth line uses \citep[\Prop 6]{vector_rademacher}. 
The fifth line follows from replacing $\bar{z}_{\theta}$ and $\tilde{z}_{\theta}$ with a supremum over $Z_{\theta}$. 
In second to last line we remove the absolute value because the same maximum will be attained by maximizing the difference. 
The last line comes from the symmetry of the Rademacher random variable.
\end{proof}
To get the final result, we use \citep[\Thm 3]{vector_rademacher} which involved a induction step to sum over the $N$ samples and Lemma~\ref{lemma:2}. 
Then, \citep[\Cor 4]{vector_rademacher} directly follows and we finish with 
\begin{align*}
\textstyle \ERad_S(\mathcal{S}) &\textstyle=\Expect_{\sigma}\left[\sup_{\hw \in \mathcal{H}} \frac{1}{N} \sum_{i=1}^N \sigma_i \Loi\left(\Aki(\hw(\theta_i))\right)\right] \\
&\textstyle\leq \pfct \Expect_{\sigma}\left[\sup_{\hw \in \mathcal{H}} \frac{1}{N} \sum_{i=1}^N \sigma_{i}^T \hw(\theta_i)\right] \\
&\textstyle= \pfct \ERad_S(\mathcal{H}).
\end{align*}
The empirical Rademacher complexity of $\mathcal{S}$ over samples $S$ is $\ERad_S(\mathcal{S})$. The inequality comes from Lemma~\ref{lemma:2}.
The last line follows from the definition of the multivariate empirical Rademacher complexity of $\mathcal{H}$.
The worst-case loss is $\pfc B$ which follows from Lemma~\ref{lemma:lemma1}.
Last, we take the expectation to get the same bound for the Rademacher complexity and use Equation~\eqref{gen_bd_rad} to finish the proof.


\acks{The author(s) are pleased to acknowledge that the work reported on in this paper was substantially performed using the Princeton Research Computing resources at Princeton University which is consortium of groups led by the Princeton Institute for Computational Science and Engineering (PICSciE) and Office of Information Technology's Research Computing.}

\bibliography{l2ws_l4dc}

\end{document}